\documentclass[a4paper,12pt]{article}
\usepackage{amsmath,amsfonts,amssymb,amsthm}
\usepackage{graphicx}
\usepackage{bm}
\usepackage{color}
\usepackage{mathtools}
\usepackage{geometry}
\usepackage{cite}

\newtheorem{example}{Example}%

\newtheorem{statement}{Statement}

\newtheorem{theorem}{Theorem}
\DeclareMathOperator{\Span}{span}
\DeclareMathOperator{\Ker}{Ker}
\DeclareMathOperator{\Det}{det}
\DeclareMathOperator{\Rank}{rank}
\DeclareMathOperator{\Dim}{dim}

\makeatletter
\renewcommand*\env@matrix[1][*\c@MaxMatrixCols c]{%
\hskip -\arraycolsep
\let\@ifnextchar\new@ifnextchar
\array{#1}}
\makeatother

\title{Algebraic Modification of the Method of Undetermined Coefficients For Solving Nonhomogeneous Linear Difference Equations}

\author{Timofey Lomonosov}

\date{%
    \textit{HSE University}\\
    \textit{Pokrovskii Bd. 11, 109028 Moscow, Russia}\\%
    \today
}

\begin{document}

\maketitle

\abstract{In this paper, an algebraic modification of the method of undetermined coefficients for solving nonhomogeneous linear stationary difference equations for quasipolynomial right-hand sides is proposed. Although the classical method of undetermined coefficients is well-known in both differential equations and difference equations case, its application in the difference equations case is severely limited. For example, it is hard to apply for rather complex expressions that can arise in case of complex quasipolynomials and resonance. The novelty of the research is the proposition of an algebraic modification to the method. That modification eliminates major drawbacks of the primary method and also allows to modify the superposition principle to apply the method to the entire difference equation at once without dividing the problem into several less complicated ones. The superposition principle in matrix form is formulated.}


\section{Introduction}\label{sec1}

The idea of representing various mathematical processes as operators has long been studied. Vast literature is dedicated to representing calculus, differentiation, and integration in such a form, see, e.g., \cite{C18,S02,F20}. There, the authors proposed a matrix differential operator method for finding a particular solution to certain differential equations, which, essentially, is a modification to a method of undetermined coefficients. With the new method, it is possible to get a system of linear algebraic equations with respect to decomposition coefficients of a quasipolynomial right-hand side of the equation without actually using the differentiation operation.

Of course, one may object that the method of undetermined coefficients is quite well described in all the classical textbooks on differential equations (see, e.g., \cite{Z11,BD12,NSS12} and question necessity of its modification. The answer to this objection follows: the method of undetermined coefficients has two significant drawbacks. First, the method is complex to apply in practice to the most interesting cases of equations with complex resonance. Second, because of its complexity, it is impossible to put those most interesting problems into the exam sheets.

Now, along differential equations, it is difference equations that are also widely used in practice. However, the difference equations, are often left behind. Nonetheless, the theory of difference equations is applicable in many areas, that require modeling of different systems. The most well-known of those models is, of course, market cobweb model. Cobweb model with inventory, for instance, is described with a second-order difference equation. Besides, applying the method of undetermined coefficients to difference equations is much more time-consuming than in the differential case, see, e.g., \cite{R01}, since in complex case we have to deal with equations of the form $\sin\alpha(n+k)$ and $\cos\alpha(n+k)$. That drastically reduces the quantity of information that one could share with students within given hours. Note that because of this, even in \cite{R01}, there are no exercises on those difference equations, for which complex eigenvalues of the operator in the left-hand side of the equation lie outside a rather limited range of arguments.

In this paper, similarly to \cite{C18,F20}, an algebraic modification of the method of undetermined coefficients is developed for solving higher-order linear difference stationary equations. Theoretical arguments for the validity of the method are more complicated compared to the differential case because of the existence of a degenerate case for difference equations. This modification is new in the area of difference equations, and it can potentially be taught to students along with more traditional approaches. That not only allows to develop their cognitive abilities, but also allows to show them the interconnection between various mathematical disciplines.

\section{Geometry of the Simplest Shift Operator}\label{sec2}

If we want students to better comprehend the underlying principles in nonhomogeneous linear difference equations with constant coefficients, then it is quite useful to pose some algebraic problems related to properties of the simplest shift operator.

At first, let $\mathbb N_0 = (0,1,2,\ldots)$ be an extended set of natural numbers. We consider the space $\mathbb C^{\mathbb N_0}$ of all \textit{complex-valued} sequences indexed from zero. Let $e_0=(1,0,0,0,\ldots)$. Then $ \Span\{e_0\}$ is a subspace of $\mathbb C^{\mathbb N_0}$. Now let $S$ be the quotient space $\mathbb C^{\mathbb N_0}\setminus\Span\{e_0\}$.

The following statements are quite easy to prove.
\begin{statement}
Any number $\lambda \in \mathbb C \setminus \{0\}$ is an eigenvalue of the operator $T$.
\end{statement}
\begin{proof}
It is well known that the initial-value problem for difference equation $$y_{n+1}=\lambda y_n, \, y_0=C$$ has a general solution $y=C\lambda^n$ for any $\lambda \in \mathbb C \setminus \{0\}$. 

That is, there exists a nonzero sequence $y_n \in S$ such that $Ty = \lambda y$.

For $\lambda = 0$ we have a degenerate equation $$y_n=0, y_0=C$$ which has infinitely many solutions of the form $\{C,0,0,\ldots\} \subset \Span\{e_0\}$. As $S$ is a quotient space by that very space, the equation in question has a unique solution $y=0$ in $S$ (i.e., it is zero everywhere except maybe the 0-th element of the sequence). Hence, $\lambda = 0$ is not the eigenvalue of the operator.
\end{proof}
\begin{statement}
Let $\lambda \in \mathbb C \setminus \{0\}$. Then the generalized eigenspace of the operator $T$ that corresponds to $\lambda$ is given by $V_\lambda = \bigcup\limits_{k=0}^{\infty} \{n^k\lambda^n\}$.
\end{statement}
\begin{proof}
Fix an arbitrary $\lambda \in \mathbb C \setminus \{0\}$. The generalized eigensequence $\{y_n\}_{n=1}^{\infty}$ of rank $m$ of the operator $T$ satisfies the equation $(T-\lambda I)^my=0$ given that for any integer $1 \leq k \leq m - 1$ $(T-\lambda I)^ky \ne 0$.

The equation $(T-\lambda I)^m y=0$ can be solved as in \cite{G71} with its fundamental system of solutions being $\{\lambda^n,n\lambda^n,\ldots,n^{m-1}\lambda^n\}$. However, only the latter sequence is a generalized eigensequence of rank $m$, as all the preceding solutions are also solutions to an equation of the same form with a lower degree of the operator on the left-hand side.
\end{proof}

Now that we know the geometry of the shift operator, we can proceed further.

\section{Algebraic Approach to Solving Linear Nonhomogeneous Difference Equations with Constant Coefficients}

Consider the space $L_\mu^s = \Span\{\mu^n,n\mu^n,\ldots,n^s\mu^n\} \subset S$ where $\mu \in \mathbb C \setminus \{0\}$ and $s \in \mathbb N$.

\begin{theorem}\label{geom}
\par 1. The restriction $T_{\mu}^s = T\rvert _{L_\mu^s}$ of the operator $T$ onto $L_{\mu}^s$ is an automorphism of the subspace $L_{\mu}^s$, and the matrix $\mathcal T_\mu^s$ of the restriction operator in the basis $\{\mu^n,n\mu^n,\ldots,n^s\mu^n\}$ has the form
$$\mathcal T_{\mu}^s = \mu \begin{pmatrix}1&1&1&\ldots&1\\0&1&2&\ldots&s \choose 1\\0&0&1&\ldots&s \choose 2\\\ldots & \ldots & \ldots & \ldots & \ldots\\0&0&0&\ldots&1\end{pmatrix}.$$

\par 2. There exists a basis in $L_\mu^s$ such that the matrix $\mathcal T_{\mu}^s$ in that basis has the form 
$$\mathcal T_{\mu}^s = \mu I_{s+1} + U_{s+1}$$
where $I_k$ is a $k$-th order identity matrix and $U_k$ is a $k$-th order upper shift matrix (i.e., $U_{kl}=\delta_{i+1,j}$ where $\delta_{ij}$ is the Kronecker delta symbol). That means that the Jordan form of the $\mathcal T_{\mu}^s$ matrix consists of a single cell.
\end{theorem}

\begin{proof}

\par 1. For $k = 0$ we simply have $(Ty)_n=\mu^{n+1}=\mu\cdot\mu^n$.

Let $y_n=n^k\mu^n$ for some natural $k$ such that $1 \leq k \leq s$. Then we have
$$(Ty)_n=(n+1)^k\mu^{n+1}=\mu \cdot (n+1)^k\mu^n = \mu \sum\limits_{j=0}^k {k \choose j} n^j \cdot \mu^n = \mu\sum\limits_{j=0}^k {k \choose j} n^j\mu^n.$$

\par 2. Note that in our case $\mathcal T_{\mu}^s = \mu P_n$, where $P_n$ is an $n \times n$ upper triangular matrix of binomial coefficients. It is well known that if $S_n$ is a $n \times n$ matrix of Stirling partition numbers, then the following identity holds \cite{C02}:
$$S_nP_nS_n^{-1} = \begin{pmatrix}
1 & 1 & 0 & \ldots & 0\\
0 & 1 & 2 & \ldots & 0\\
\ldots & \ldots & \ldots & \ldots & \ldots\\
0&0&0&\ldots&s\\
0&0&0&\ldots&1
\end{pmatrix}.$$

Now, $S_n^{-1}P_nS_n$ is similar to a Jordan matrix via $$DS_n^{-1}P_nS_nD^{-1}=J(P_n)$$ where $D=\operatorname{diag}\{0!,1!,2!,\ldots,s!\}$. Hence, the Jordan form of $P_n$ consists of a single block with $\lambda = 1$.

Therefore, the matrix $\mu P_n$ consists of a single block with $\lambda = \mu$.

\end{proof}

For any complex $\mu=\lvert\mu\rvert e^{i\varphi}$ that has a nonzero imaginary part, let
\begin{gather*}L_{\mu,C}^s = \Span\{\lvert\mu\rvert^n\sin n\varphi, \lvert\mu\rvert^n\cos n\varphi, n\lvert \mu \rvert ^n\sin n\varphi, \\ n\lvert\mu\rvert^n\cos n\varphi,\ldots,n^s\lvert \mu \rvert^n\sin n\varphi, n^s\lvert \mu \rvert^n \cos n\varphi\}\end{gather*} be the space optained from $L_\mu^s$ by restricting the scalars to the real field and the elements to the real-valued sequences, and changing its basis. Let also $R \coloneqq T^n+a^{(1)}T^{n-1}+\ldots+a^{(k)}I$ be the difference operator on the left-hand side.

\begin{theorem}
\par 1. The restriction $T_{\mu,C}^s = T\rvert _{L_{\mu,C}^s}$ of the operator $T$ onto $L_{\mu,C}^s$ is an automorphism of the subspace $L_{\mu,C}^s$, and the matrix $\mathcal T_{\mu,C}^s$ of the restriction operator in the basis $$\{\lvert\mu\rvert^n\sin n\varphi, \lvert\mu\rvert^n\cos n\varphi, n\lvert \mu \rvert ^n\sin n\varphi, \\ n\lvert\mu\rvert^n\cos n\varphi,\ldots,n^s\lvert \mu \rvert^n\sin n\varphi, n^s\lvert \mu \rvert^n \cos n\varphi\}$$ has the form
$$\mathcal T_{\mu,C}^s = \mu \begin{pmatrix}J&J&J&\ldots&J\\0&J&2J&\ldots&{s \choose 1}J\\0&0&R&\ldots&{s \choose 2}J\\\ldots & \ldots & \ldots & \ldots & \ldots\\0&0&0&\ldots&J\end{pmatrix}$$
where $J=\begin{pmatrix}\cos\varphi&-\sin\varphi\\\sin\varphi&\cos\varphi\end{pmatrix}$.

\par 2. There exists a basis in $L_{\mu,C}^s$ such that the matrix $\mathcal T_{\mu}^s$ in that basis has the form 
$$\mathcal T_{\mu,C}^s = I_{s+1} \otimes J + U_{s+1} \otimes I_2$$
where $\otimes$ is the Kronecker matrix product.
\end{theorem}

\begin{proof}
\par 1. For the sequence $y_n=|\mu|^n\sin n\varphi$,  we just have \begin{gather*}(Ty)_n=\mu^{n+1}\sin(n+1)\varphi=|\mu|\cdot|\mu|^n(\cos\varphi\sin n\varphi + \sin\varphi\cos n\varphi)=\\=|\mu|\left(\cos \varphi |\mu|^n\sin n\varphi + \sin\varphi |\mu|^n\cos n\varphi\right),\end{gather*}
and, for the sequence $y_n=|\mu|^n\sin n\varphi$, we have \begin{gather*}(Ty)_n=\mu^{n+1}\cos(n+1)\varphi=|\mu|\cdot|\mu|^n(-\sin\varphi\sin n\varphi + \cos\varphi\cos n\varphi)=\\=|\mu|\left(-\sin \varphi |\mu|^n\sin n\varphi + \cos\varphi |\mu|^n\cos n\varphi\right).\end{gather*}

Let $y_n=n^k|\mu|^n\sin n\varphi$ for some $k$ such that $1 \leq k \leq s$. Then we have
\begin{gather*}y_{n+1}=(n+1)^k|\mu|^n\sin (n+1)\varphi = |\mu|\sum\limits_{j=0}^k {k \choose j} n^j |\mu|^n(\cos\varphi \sin n\varphi + \sin \varphi \cos n\varphi)=
\\=|\mu|\sum\limits_{j=0}^k \left[\left({n\choose k}\cos\varphi\right)n^j|\mu|^n\sin n\varphi + \left({n\choose k}\sin\varphi\right)n^j|\mu|^n\cos n\varphi\right].\end{gather*}

We get a similar result for $y_n=n^k|\mu|^n\cos n\varphi$ for some $k$ such that $1 \leq k \leq s$. Then we have
\begin{gather*}y_{n+1}=(n+1)^k|\mu|^n\cos (n+1)\varphi = |\mu|\sum\limits_{j=0}^k {k \choose j} n^j |\mu|^n(-\sin\varphi \sin n\varphi + \cos \varphi \cos n\varphi)=
\\=|\mu|\sum\limits_{j=0}^k \left[-\left({n\choose k}\sin\varphi\right)n^j|\mu|^n\sin n\varphi + \left({n\choose k}\cos\varphi\right)n^j|\mu|^n\cos n\varphi\right].\end{gather*}

\par 2. Note that in our case $\mathcal T_{\mu,C}^s = |\mu| (P_n \otimes J)$, where $P_n$ is an $n \times n$ upper triangular matrix of binomial coefficients. Using the proven result in Theorem \ref{geom} above, we conclude that the spectrum of the matrix $P_n$ consists of a single number $1$, and the spectrum of the matrix $J$ consists of two eigenvalues $e^{i\varphi}$ and $e^{-i\varphi}$. Then by the property of the Kronecker product of two matrices, the spectrum of the matrix $\mathcal T_{\mu,C}^s$ consists of two eigenvalues $|\mu|e^{i\varphi}$ and $|\mu|e^{-i\varphi}$. So, the Jordan canonical form for $\mathcal T_{\mu,C}^s$ has the form $\begin{pmatrix}J_k((|\mu|e^{i\varphi})&0\\0&J_k((|\mu|e^{-\varphi})\end{pmatrix}$, where $J_k(\lambda)$ is a $k$-th order Jordan cell that corresponds to the eigenvalue $\lambda$.

This Jordan matrix is permutation similar \cite{H88} to the block matrix
$$\begin{pmatrix}D(|\mu|e^{i\varphi})&I&O&\ldots&O\\O&D((|\mu|e^{i\varphi})&I&\ldots&O\\\ldots&\ldots&\ldots&\ldots&\ldots\\O&O&O&\ldots&I\\O&O&O&\ldots&D((|\mu|e^{i\varphi})\end{pmatrix},$$
where $D(|\mu|(e^{i\varphi})=\begin{pmatrix}(|\mu|e^{i\varphi}&0\\0&(|\mu|e^{-i\varphi}\end{pmatrix}.$

Also, each of the blocks $D((|\mu|e^{i\varphi})$ is similar to a real matrix
$$(|\mu|J=|\mu|\begin{pmatrix}\cos \varphi&-\sin\varphi\\\sin\varphi&\cos\varphi\end{pmatrix},$$
thus, the block pair of $k \times k$ conjugate Jordan cells with complex $\lambda=e^{i\varphi}$ is similar to a $2k\times2k$ matrix $ |\mu|I_{s+1} \otimes J + U_{s+1} \otimes I_2.$ 
\end{proof}

\subsection{Algebraic Theorems for Difference Equations}

Let $S_\mathbb R$ be a subspace of $S$ that consists only of real-valued sequences and is a vector space over a field of real numbers. We consider a higher-order nonhomogeneous linear ordinary difference equation (ODE) with constant coefficients

\begin{gather}\label{lnde}
y_{n+k}+a^{(1)}y_{n+k-1}+\ldots+a^{(k)}y_n=f_n
\end{gather}
where $y \in S_\mathbb R$ is a sought sequence, $a^{(1)},\ldots, a^{(k)}$ are given numbers with $a^{(k)} \ne 0$, and $f \in S_\mathbb R$ is a given sequence.

\begin{theorem} In \eqref{lnde}, let $\mu$ be a real number and $f \in L_\mu^s$
\par 1. If $\mu$ is not a characteristic number of equation \eqref{lnde}, then a particular solution $y^{(0)} \in L_\mu^s$, and its coefficients in the basis $\{\mu^n, n\mu^n,\ldots,n^s\mu^n\}$ can be found from the system of equations
$$R_\mu^s\*y = \*f,$$
where $R_\mu^s=(\mathcal T_\mu^s)^n+a_1(\mathcal T_\mu^s)^{n-1}+\ldots+a_nI$ and $\*f$ is a column-vector obtained from expansion of the right-hand side $f$ in the equation into the given basis, and the system in question has a unique solution.

\par 2. If $\mu$ is a characteristic number of equation \eqref{lnde} and is of multiplicity $m$, then the particular solution $y^{(0)} \in L_{\mu}^{m+s}=K_\mu^m \oplus L_\mu^{m,s}$, where
$$K_\mu^m=\Span\left\{\mu^n,n\mu^m\ldots,n^{m-1}\mu^n\right\} = \Ker R\rvert _{L_\mu^{m+s}}$$
and
$$L_\mu^{m,s}=\Span\left\{n^m\mu^n,n^{m+1}e^{\mu n},\ldots,n^{m+s}e^{\mu n}\right\}.$$

In this case we can find the projection $\*y_p$ of the particular solution $y^{(0)}$ onto $L_{\mu}^{m,s}$, and the coefficients of that projection in the given basis from the system of equations $$R_\mu^{m,s}\*y_p=\*f,$$
where $R_\mu^{m,s}$ is a matrix obtained by eliminating the first $m$ columns and the last $m$ columns from the matrix $(T_\mu^s)^n+a_1(T_\mu^s)^{n-1}+\ldots+a_nI$ and $\*f$ is a column-vector obtained by expanding the right-hand side $f \in L_\mu^s$ in the basis $\{\mu^n, n\mu^n,\ldots,n^s\mu^n\}$. The system in question has a unique solution.
\end{theorem}

\begin{proof}
For simplicity, we will prove the theorem in the Jordan basis of the operator $T_\mu^s$ in which its matrix is given by the formula $\mathcal T_\mu^s = \mu I_{s+1} + U_{s+1}$.

\par The difference operator $R$ on the left-hand side of equation \eqref{lnde} can be factorized as
$$R=\left(T-\mu_1 I\right)^{s_1}\left(T-\mu_2 I\right)^{s_2}\ldots\left(T-\mu_l I\right)^{s_l}$$
where $\mu_1,\ldots,\mu_l$ are the characteristic numbers of the homogeneous equation and $s_1,\ldots,s_l$ are their multiplicities, respectively.

\par 1. If neither of $\mu_1,\ldots,\mu_l$ equals $\mu$, then each of the operators $\left(T-\mu_k I\right)^{s_k}\rvert _{L_{\mu}^s}$ has the operator matrix
$$(T_\mu^s-\mu_k I)^{s_k}=\left[(\mu - \mu_k)I_{s+1}+ \begin{pmatrix}\*0_s & I_s\\0&\*0_s^T\end{pmatrix}\right]^{s_k}.$$

The determinant of such a matrix is not zero for any of $\mu_1,\ldots,\mu_l$, therefore $\Det R_\mu^s \ne 0$. Thus, the system $R_\mu^s \*y = \*f$ has a unique solution.

\par 2. Now suppose that there is $\mu$ among $\mu_1,\ldots,\mu_l$. Without loss of generality, we may assume that $\mu_1=\mu$ and $s_1=m$.

Then the operator $\left(T-\mu\right)^{s_1}\rvert _{L_\mu^{s_1+s}}$ has the operator matrix
$$(T_\mu^{s_1+s}-\mu I)^{s_1}=U_s^{s_1}=\begin{pmatrix}O_{(s+1) \times s_1} & I_{s+1}\\O_{s_1\times s_1}&O_{s_1 \times (s+1)}\end{pmatrix}.$$

Then, since all of the matrices $\left(T_\mu^{s_1+s}-\mu_k I\right)^{s_k}$, $k=2,\ldots,n$ are upper triangular and $(T_\mu^{s_1+s}-\mu I)^{s_1}$ is strictly upper triangular, the product $$\left(T_\mu^{s_1+s}-\mu_2 I\right)^{s_2}\left(T_\mu^{s_1+s}-\mu_3 I\right)^{s_3}\ldots\left(T_\mu^{s_1+s}-\mu_l I\right)^{s_l}$$ is strictly upper triangular. Moreover, it has a form $$\begin{pmatrix}O_{(s+1)\times s_1}& F_{(s+1)\times(s+1)}\\O_{s_1 \times s_1} & O_{s_1 \times (s+1)}\end{pmatrix}$$ for some $(s+1)\times (s+1)$ matrix $F$.Now we show that the matrix $F$ is non-degenerate. Note that the subspace $K_\mu^{s_1+s}$ has dimension of $s_1$ and is the kernel of $R\rvert _{L_\mu^{s_1+s}}$. By the rank-nullity theorem we have $$\Rank R\rvert _{L_{\mu}^{s_1+s}}=\Rank F = \Dim L_{\mu}^{s_1+s}-\Dim \Ker R\rvert _{L_{\mu}^{s_1+s}}=(s_1+s+1)-s_1=s+1.$$

Thus, $F$ is indeed non-degenerate, and the system of equations $F\*y_p = \*f$ has a unique solution $\*y_p \in \mathbb R^{s+1} \cong L_{\mu}^{s_1,s}$.
\end{proof}

\begin{theorem}\label{resonance} In \eqref{lnde}, let $f \in L_{\mu,C}^{s}$.
\par 1. If $\mu$ is not a characteristic number of equation \eqref{lnde}, then a particular solution $y^{(0)} \in L_\mu^{C,s}$, and its coefficients in the basis $$\left\{\lvert \mu \rvert ^n\sin n\varphi, \lvert \mu \rvert ^n\cos\varphi, n\lvert \mu \rvert ^n\sin\varphi,n\lvert \mu \rvert ^m\cos n\varphi,\ldots,n^s\lvert \mu \rvert ^n \sin n\varphi, n^s\lvert \mu \rvert ^n\cos n\varphi\right\}$$ can be found from the system of equations
$$R_{\mu,C}^s\*y = \*f,$$
where $R_{\mu,C}^s=(T_{\mu,C}^s)^n+a_1(T_{\mu,C}^s)^{n-1}+\ldots+a_nI$ and $\*f$ is a column-vector obtained from expansion of the right-hand side $f$ in the equation into the given basis, and the system in question has a unique solution.

\par 2. If $\mu$ is a characteristic number of equation \eqref{lnde} and is of multiplicity $m$, then the particular solution $y^{(0)} \in L_{\mu,C}^{m+s}=K_{\mu,C}^m \oplus L_{\mu,C}^{m,s}$, where
\begin{gather*}K_{\mu,C}^m=\Span\left\{\lvert \mu \rvert ^n\sin n\varphi, \lvert \mu \rvert ^n\cos\varphi, n\lvert \mu \rvert ^n\sin\varphi,n\lvert \mu \rvert ^m\cos n\varphi,\ldots,\right.\\\left.\ldots, n^{m-1}\lvert \mu \rvert ^n \sin n\varphi, n^{m-1}\lvert \mu \rvert ^n\cos n\varphi\right\} = \Ker R\rvert _{L_{\mu,C}^{m+s}}
\end{gather*}
and
\begin{gather*}
L_\mu^{m,s}=\Span\left\{n^m\lvert \mu \rvert ^n\sin n\varphi, n^m\lvert \mu \rvert ^n\cos\varphi, n^{m+1}\lvert \mu \rvert ^n\sin\varphi,n^{m+1}\lvert \mu \rvert ^m\cos n\varphi,\ldots,\right.\\\left.\ldots, n^{m+s}\lvert \mu \rvert ^n \sin n\varphi, n^{m+s}\lvert \mu \rvert ^n\cos n\varphi\right\}.
\end{gather*}

In this case we can find the projection $\*y_p$ of the particular solution $y^{(0)}$ onto $L_{\mu,C}^{m,s}$, and the coefficients of that projection in the given basis from the system of equations $$R_{\mu,C}^{m,s}\*y_p=\*f,$$
where $R_{\mu,C}^{m,s}$ is a matrix obtained by eliminating the first $m$ columns and the last $m$ columns from the matrix $(T_{\mu,C}^s)^n+a_1(T_{\mu,C}^s)^{n-1}+\ldots+a_nI$ and $\*f$ is a column-vector obtained by expanding the right-hand side $f \in L_{\mu,C}^s$ in the basis $$\left\{\lvert \mu \rvert ^n\sin n\varphi, \lvert \mu \rvert ^n\cos\varphi, n\lvert \mu \rvert ^n\sin\varphi,n\lvert \mu \rvert ^m\cos n\varphi,\ldots,n^s\lvert \mu \rvert ^n \sin n\varphi, n^s\lvert \mu \rvert ^n\cos n\varphi\right\}$$ The system in question has a unique solution.
\end{theorem}

\begin{proof}

As in the previous case, for simplicity, we will prove the theorem in the Jordan basis of the operator $T_\mu^s$ in which its matrix is given by the formula $\mathcal T_{\mu,C}^s = I_{s+1} \otimes J + U_{s+1} \otimes I_2$, where $\otimes$ is the Kronecker matrix product.

\par The difference operator $R$ on the left-hand side of equation \eqref{lnde} can be factorized as
$$R=\left(T-\mu_1 I\right)^{s_1}\left(T-\mu_2 I\right)^{s_2}\ldots\left(T-\mu_l I\right)^{s_l}$$
where $\mu_1,\ldots,\mu_l$ are the characteristic numbers of the homogeneous equation and $s_1,\ldots,s_l$ are their multiplicities, respectively.

\par 1. If neither of $\mu_1,\ldots,\mu_l$ equals $\mu$, then each of the operators $\left(T-\mu_k I\right)^{s_k}\rvert _{L_{\mu,C}^s}$ has the operator matrix
$$(T_{\mu,C}^s-\mu_k I)^{s_k}=\left[I_{s+1} \otimes (J-\mu_k I_2)+U_{2(s+1)}\right]^{s_k}.$$
The determinant of such a matrix is not zero for any of $\mu_1,\ldots,\mu_l$, therefore $\Det R_{\mu,C}^s \ne 0$. Thus, the system $R_{\mu,C}^s \*y = \*f$ has a unique solution.

\par 2. Now suppose that there is $\mu$ among $\mu_1,\ldots,\mu_l$. Recall that since all the coefficients $a_1,\ldots,a_n$ in equation \eqref{lnde} are real, then complex roots always appear as a pair of complex conjugate numbers. Without loss of generality, we may assume that $\mu=\mu_1$ and $m=s_1$. We first consider the matrix
\begin{gather*}(T_{\mu,C}^{s_1+s}-\mu I)(T_{\mu,C}^{s_1+s}-\bar{\mu}I)=\\=\left[I_{s_1+s+1} \otimes \beta\begin{pmatrix}-\*i&1\\-1&-\*i\end{pmatrix}+ U_{s_1+s+1}\otimes I_2\right]\left[I_{s+1} \otimes \beta\begin{pmatrix}\*i&1\\-1&\*i\end{pmatrix})+ U_{s+1}\otimes I_2\right]=\\
=\beta^2(I_{s_1+s+1} \otimes O_2)+\beta\left[U_{s_1+s+1}\otimes\begin{pmatrix}0&2\\-2&0\end{pmatrix}\right]+U_{s_1+s+1}^2\otimes I_2=\\
=U_{s_1+s+1}\otimes B+U_{s_1+s+1}^2\otimes I_2,
\end{gather*}
where $B=\begin{pmatrix}0&2\beta\\-2\beta&0\end{pmatrix}$.

Note that matrices $U_{s_1+s+1}\otimes B$ and $U_{s_1+s+1}^2 \otimes I_2$ commute. Therefore, we have

\begin{gather*}
\left[(T_{\mu,C}^{s_1+s}-\mu I)(T_{\mu,C}^{s_1+s}-\bar{\mu}I)\right]^{s_1}=\left(U_{s_1+s+1} \otimes B + U_{s_1+s+1}^2 \otimes I_2\right)^{s_1}=\\
=\sum\limits_{k=0}^{s_1} {s_1 \choose k} (U_{s_1+s+1}\otimes B)^{k}(U_{s_1+s+1}^2 \otimes I_2)^{s_1-k}=\sum\limits_{k=0}^{s_1}{s_1 \choose k}(U_{s_1+s+1}^{s_1+k} \otimes B^{k})= \\ = \begin{pmatrix}O_{2(s+1) \times 2s_1} & C_{2(s+1) \times 2(s+1)}\\O_{2s_1 \times 2s_1}&O_{2s_1\times 2(s+1)}\end{pmatrix}
\end{gather*}
where $C$ is a $2(s+1) \times 2(s$ matrix.


Since all of the matrices $\left(T_\mu^{s_1+s}-\mu_k I\right)^{s_k}$, $k=2,\ldots,n$ are upper triangular and the product $\left[(T_{\mu,C}^s-\mu I)(T_{\mu,C}^s-\bar{\mu}I)\right]^{s_1}$ is strictly upper triangular, the product $$\left(T_\mu^{s_1+s}-\mu_2 I\right)^{s_2}\left(T_\mu^{s_1+s}-\mu_3 I\right)^{s_3}\ldots\left(T_\mu^{s_1+s}-\mu_l I\right)^{s_l}$$ is strictly upper triangular. Moreover, it has a form $$\begin{pmatrix}O_{2(s+1) \times 2s_1} & C_{2(s+1) \times 2(s+1)}\\O_{2s_1 \times 2s_1}&O_{2s_1\times 2(s+1)}\end{pmatrix}$$ for some $2(s+1)\times 2(s+1)$ matrix $\tilde C$.

It follows from the rank-nullity theorem that $\tilde C$ is non-degenerate.

\end{proof}

\section{Demonstration of the Approach}

In this section we are going to compare the proposed method and the well-known method of undetermined coefficients applied to two difference equations, each of which has a complex quasi-polynomial on the right-hand side.

\begin{example}
Find a particular solution of the equation $y_{n+2}+y_n=\cos \frac{\pi n}{8}$.
\end{example}

At first, we are going to solve this exercise using the proposed method. We see that there is no resonance in this equation as the characteristic numbers of the homogeneous equation are $\pm \mathbf i$, where $\mathbf i$ is the imaginary unit, and the right-hand side is a quasi-polynomial that corresponds to a pair of conjugate numbers $e^{\pm \mathbf i \frac{\pi}{8}}$.

The matrix of the operator $T\rvert_{L_{1,\frac{\pi}{8}}^1}$ has the form $\mathcal T_{1,\frac{\pi}{8}}^1=\begin{pmatrix}\cos\frac{\pi}{8}&-\sin\frac{\pi}{8}\\\sin\frac{\pi}{8}&\cos\frac{\pi}{8}\end{pmatrix}$.

The matrix of the operator $R_{{1,\frac{\pi}{8}}^1}$ has the form
$$(T_{1,\frac{\pi}{8}}^1)^2+I=\begin{pmatrix}\frac{1}{\sqrt2}+1&-\frac{1}{\sqrt 2}\\\frac{1}{\sqrt 2}&\frac{1}{\sqrt 2}+1\end{pmatrix}.$$

Now, using elementary transformations, we get
\begin{gather*}\begin{pmatrix}[cc|c]\frac{\sqrt 2 + 1}{\sqrt 2}&-\frac{1}{\sqrt 2}&0\\[1mm]\frac{1}{\sqrt 2}&\frac{\sqrt 2 +1}{\sqrt 2}&1\end{pmatrix}\sim\begin{pmatrix}[cc|c]1&\sqrt 2 +1&\sqrt 2\\[1mm]\sqrt 2 +1&-1&0\end{pmatrix}\sim\\[1mm]\sim\begin{pmatrix}[cc|c]1&\sqrt 2 +1 & \sqrt 2\\[1mm]0 & -4-2\sqrt 2&-2-\sqrt 2\end{pmatrix}\sim\begin{pmatrix}[cc|c]1&\sqrt 2 +1 & \sqrt 2\\[1mm]0&1&\frac{1}{2}\end{pmatrix}\sim\begin{pmatrix}[cc|c]1&0&\frac{\sqrt 2-1}{2}\\[1mm]0&1&\frac{1}{2}\end{pmatrix}.\end{gather*}

Therefore the sought particular solution is $y_n=\frac{\sqrt 2-1}{2}\sin\frac{\pi n}{8}+\frac{1}{2}\cos\frac{\pi n}{8}.$

How would we solve the same equation using the method of undetermined coefficients? Of course, we would seek the particular solution in the form $y_n=A\sin\frac{\pi n}{8}+B\cos\frac{\pi n}{8}$ and then substitute $y_n$ into the equation. We would get
$$A\sin\frac{\pi(n+2)}{8}+B\cos\frac{\pi(n+2)}{8}+A\sin\frac{\pi n}{8}+B\cos\frac{\pi n}{8}=\cos\frac{\pi n}{8}.$$

Eventually, we would get the same system of linear equations, but it would take more time and students would be more prone to making an arithmetic error, which would void their entire solution.

Now let us compare our approaches in a resonant case.

\begin{example}
Find a particular solution of the equation $y_{n+2}+y_n=\sin \frac{\pi n}{2}$.
\end{example}

We see that there is a resonance on the right-hand side $f \in L_{1,\frac{\pi}{2}}^{1+1}$. The matrix of the operator $T\rvert_{ L_{1,\frac{\pi}{2}}^{1+1}}$ has the form
$$ \mathcal T_{1,\frac{\pi}{2}}^{1+1}=\begin{pmatrix}0&-1&0&-1\\1&0&1&0\\0&0&0&-1\\0&0&1&0\end{pmatrix}.$$

The matrix of the operator $R\rvert_{ L_{1,\frac{\pi}{2}}^{1+1}}$ assumes the form$$ R_{1,\frac{\pi}{2}}^{1+1}=\begin{pmatrix}0&0&-2&0\\0&0&0&-2\\0&0&0&0\\0&0&0&0\end{pmatrix}.$$

Now, using elementary transformations, we get

\begin{gather*}\begin{pmatrix}[cc|c]-2&0&1\\[1mm]0&-2&0\end{pmatrix}\sim\begin{pmatrix}[cc|c]1&0&-\frac{1}{2}\\[1mm]0&1&0\end{pmatrix}.\end{gather*}

Thus, the sought particular solution is $y_n=-\frac{1}{2}n\sin\frac{\pi n}{2}$.

Using the method of undetermined coefficients, we would seek the particular solution in the form $y_n=An\sin\frac{\pi n}{8}+Bn\cos\frac{\pi n}{8}$ and then substitute $y_n$ into the equation.

We would get
$$A(n+2)\sin\frac{\pi(n+2)}{2}+B(n+2)\cos\frac{\pi(n+2)}{2}+An\sin\frac{\pi n}{2}+Bn\cos\frac{\pi n}{2}=\cos\frac{\pi n}{2}.$$

Eventually, we would get the same system of linear equations but it would take even more time to solve than the previous equation, and the probability of making an arithmetic error would be far higher than in the previous equation.

\section{The superposition principle}

The proposed modification can be extended to the problems that are solved by the superposition principle in the classical method of undetermined coefficients. If we use the algebraic modification presented in this paper we no longer need to divide the problem into several smaller ones.

\begin{theorem}[The superposition principle]
In equation \ref{lnde},  let the right-hand side $f \in \mathcal L\{L_{\mu_1}^{s_1},\ldots,L_{\mu_k}^{s_k},L_{\mu_{k+1},C}^{s_{k+1}},\ldots,L_{\mu_l,C}^{s^l}\}$, where $k$ and $l$ are integers, $\mu_1,\ldots,\mu_k$ are reals, and $\mu_{k+1},\ldots,\mu_l$ are complex numbers.

Then, in order to find the solution of equation \ref{lnde}, we need to find the matrix of the operator $T$ in the basis of the space $\mathcal L\{L_{\mu_1}^{s_1},\ldots,L_{\mu_k}^{s_k},L_{\mu_{k+1},C}^{s_{k+1}},\ldots,L_{\mu_l,C}^{s^l}\}$, which has the form
$$\mathcal T = \begin{pmatrix}T_{\mu_1}^{s_1}&O&\ldots&\ldots&\ldots&O\\
O&\ddots&\ldots&\vdots&\ldots&\vdots\\
\vdots&\ldots&T_{\mu_k}^{s_k}&O&\ldots&\vdots\\
\vdots&\ldots&O&T_{\mu_{k+1},C}^{s_{k+1}}&\ldots&\vdots\\
\vdots&\ldots&\ldots&\ldots&\ddots&\ldots\\
O&\ldots&\ldots&\ldots&\ldots&T_{\mu_l,C}^{s_l}\end{pmatrix},$$
construct the matrix $\mathcal R=\mathcal T^n+a^{(1)}\mathcal T^{n-1}+\ldots+a^{(k)}I$, then, crossing out zero columns and rows from the matrix $\mathcal R$, get the nondegenerate matrix $\mathcal R_{nd}$, and, by solving the equation $\mathcal R_{nd}\mathbf y = f$, where $f $ is a column vector, obtained by decomposition of the right-hand side in the basis of the space  $ \mathcal L\{L_{\mu_1}^{s_1},\ldots,L_{\mu_k}^{s_k},L_{\mu_{k+1},C}^{s_{k+1}},\ldots,L_{\mu_l,C}^{s^l}\}$, find the sought particular solution.
\end{theorem}

\begin{proof} The theorem follows from the fact that the operator $T$ is an automorphism in each of the spaces $L_{\mu_i}^{s_i},\,i=1,\ldots,k$ and $L_{\mu_j,C}^{s_j},\,j=k+1,\ldots,l$,and from the representation of $\mathcal L\{L_{\mu_1}^{s_1},\ldots,L_{\mu_k}^{s_k},L_{\mu_{k+1},C}^{s_{k+1}},\ldots,L_{\mu_l,C}^{s^l}\}$ as a direct sum of the spaces included in it.
\end{proof}

\section{Conclusion}

An algebraic approach on solving linear nonhomogeneous difference equations with constant coefficients is proposed. It can be used in any difference equations course as a way of showing that there can always be more than one way of solving a problem. The proposed method clearly shows the students how various mathematical disciplines are connected as the proposed approach uses mechanisms of linear algebra that the students should have been studying earlier.

The proposed approach can also be applied to differential equations course and requires just a slight adjustment. Also this approach can lead to better implementation of difference equation solvers.

\section*{Declarations}

\noindent \textbf{Conflict of Interest.} The author declares no competing interests.

\noindent\textbf{Funding.} This article is an output of a research project implemented as part of the Basic Research Program
at the National Research University Higher School of Economics (HSE University).



\begin{thebibliography}{9}

\bibitem{C18} Chen, W. Differential Operator Method of Finding A Particular Solution to An Ordinary Nonhomogeneous Linear Differential Equation with Constant Coefficients // arXiv:1802.09343v1. https://doi.org/10.48550/arXiv.1802.09343
\bibitem{S02} Spiegel, M. R. Schaum's Outline of Theory and Problems of Advanced Mathematics for Engineers and Scientists. McGraw-Hill, 2002.


\bibitem{F20} Fecenko, J. Matrix Differential Operator Method of Finding a Particular Solution to a Nonhomogeneous Linear Ordinary Differential Equation with Constant Coefficients. // arXiv:2101.02037v1. https://doi.org/10.48550/arXiv.2101.02037

\bibitem{Z11} Zill, D.G. A First Course in Differential Equations with Modelling Applications,10th edition. Brooks/Cole, 2011.

\bibitem{BD12} Boyce, W.E. and DiPrima, R.C. Elementary Differential Equations and Boundary Value Problems, 10th edition, Wiley, 2012.

\bibitem{NSS12} Nagle, R, Saff E.B. and Suider A.D. Fundamentals of Differential Equations, 8th edition, AddisonWesley, 2012.

\bibitem{R01} Romanko, V. K. Course of Difference Equations (in Russian). Moscow, Fizmatlit Publ., 2012.


\bibitem{G71} Gelfond, A.O. Calculus of finite differences. Hindustan Publ. Company, 1971.

\bibitem{C02} Callan, D. Jordan and Smith forms of Pascal-related matrices // arXiv:math/0209356v1. https://doi.org/10.48550/arXiv.math/0209356.

\bibitem{H88} Horn, R., Johnson, C. Matrix Analysis. Cambridge University Press, ISBN 978-0-521-38632-6




\end{thebibliography}


\end{document}